\documentclass[accepted]{uai2021} 

\usepackage[american]{babel}
\usepackage{natbib} 
\usepackage{mathtools} 
\usepackage{booktabs} 
\usepackage{tikz} 

\usepackage{multirow}
\usepackage{latexsym}
\usepackage{amsmath,amssymb,amsthm,graphicx}
\usepackage{dsfont}
\usepackage{enumitem}
\theoremstyle{plain}
\usepackage[utf8]{inputenc}

\theoremstyle{plain}
\newtheorem{theorem}{Theorem}

\theoremstyle{definition}

\theoremstyle{remark}

\newtheorem*{remark}{Remark}

\newcommand{\R}{\mathbb{R}} 
\newcommand{\E}{\mathbb{E}} 
\newcommand{\V}{\text{Var}} 

\newcommand{\soo}{\Sigma_{11}}
\newcommand{\sot}{\Sigma_{12}}
\newcommand{\stt}{\Sigma_{22}}
\newcommand{\saa}{\Sigma_{aa}}
\newcommand{\M}{\mathcal{M}}

\newcommand{\ND}{\text{N}} 

    \def\cd{\stackrel{\mathcal{D}}{\longrightarrow}}

\title{Confidence in Causal Discovery with Linear Causal Models}

\author[1]{\href{mailto:David Strieder <david.strieder@tum.de>?Subject=Your UAI 2021 paper}{David Strieder}{}} 
\author[1,2]{Tobias Freidling}
\author[1]{Stefan Haffner}
\author[1]{\href{mailto:Mathias Drton <mathias.drton@tum.de>?Subject=Your UAI 2021 paper}{Mathias Drton}{}}

\affil[1]{%
    Chair of Mathematical Statistics\\
	Technical University of Munich\\
	Germany
}
\affil[2]{%
    DPMMS\\
    University of Cambridge\\
    UK
}

\begin{document}
\maketitle

\begin{abstract}
  Structural causal models postulate noisy functional relations among a set of interacting variables. The causal structure underlying each such model is naturally represented by a directed graph whose edges indicate for each variable which other variables it causally depends upon. Under a number of different model assumptions, it has been shown that this causal graph and, thus also, causal effects are identifiable from mere observational data. For these models, practical algorithms have been devised to learn the graph. Moreover, when the graph is known, standard techniques may be used to give estimates and confidence intervals for causal effects. We argue, however, that a two-step method that first learns a graph and then treats the graph as known yields confidence intervals that are overly optimistic and can drastically fail to account for the uncertain causal structure. To address this issue we lay out a framework based on test inversion that allows us to give confidence regions for total causal effects that capture both sources of uncertainty: causal structure and numerical size of nonzero effects. Our ideas are developed in the context of bivariate linear causal models with homoscedastic errors, but as we exemplify they are generalizable to larger systems as well as other settings such as, in particular, linear non-Gaussian models.
\end{abstract}

\section{Introduction}\label{sec:intro}

Anticipating the cause and effect of actions is a task the human brain is able to master every day, yet it is challenging to devise statistical methods that reliably infer cause-effect relations from available data.  The field of causal discovery seeks to address this challenge by clarifying when it is theoretically possible to infer a causal effect and by developing practical methods to estimate the effect from data \citep{Pearl09,spirtes:book}.  A widely studied approach adopts the paradigm of structural causal models, in which each variable is a function of a subset of other variables (its causes) and a stochastic error term; see also~\cite{peters:book} or \citet[Part IV]{handbook}.  The causal perspective results from viewing these functions as mechanisms that assign values based on the values of causes.  Algorithms for causal discovery infer the structure of such a causal model, which is naturally represented by a directed graph whose edges point from causes to effects.  Moreover, for a known causal graph, standard methods give point estimates and confidence intervals for a causal effect of interest;  at least, this is the case in the causally sufficient setting, where all relevant variables have been observed.

It is straightforward to combine a method that learns a graph with a method to subsequently estimate and make confidence statements about causal effects.  However, such a two-step approach tacitly conditions away the uncertainty that arises from the data-driven model choice and ignores the uncertainty that exists with respect to the causal structure.  As a result, this approach is overly optimistic in its conclusions about existence and strength of causal effects.  Despite the extensive literature that exists on causal discovery, we are not aware of prior work that accounts for uncertainty in the causal structure when providing confidence statements about the inferred causal direction or effect size.  Here, the term `confidence' is used in the technical sense of confidence sets with a given desired frequentist coverage probability.  An entirely different approach that we do not consider here would be to form Bayesian credible sets.  Indeed, a number of authors have pursued Bayesian approaches, but primarily with a focus on the uncertainty in the graphical structure as opposed to causal effects; see
 \cite{10.5555/1795114.1795143},  \cite{10.5555/3020652.3020677}, and \cite{cao:2019} for three selected examples.
 
In this article we present a framework to construct confidence intervals for the total causal effect that account for the uncertainty inherent in the data-driven selection of a causal model.  The setting we focus on pertains to the case in which we only have access to observational data but consider restricted structural causal models for which causal structure and effects are nevertheless identifiable.  Specifically, we consider the simplest such setting, namely, linear structural equation models with errors that are homoscedastic, i.e., of equal variance \citep{PetersEV}.  Their causal ordering is identifiable as it corresponds to an ordering of conditional variances \citep{DrtonEV,Ghoshal2018}. Our presentation focuses on this setting and also the two-variable case, but the framework we lay out in Section~\ref{sec:bivariate} is general and can be applied to larger systems as well as other model classes, in particular, linear models with non-Gaussian errors, as we also discuss in Sections~\ref{sec:higher-dim}-\ref{sec:discussion}.

\section{Background}
\label{sec:background}

This section reviews linear structural equation models and the total causal effect that is the central object of study.
Moreover, we discuss the difficulties of applying resampling methods for statistical inference and review the test-inversion approach.

\subsection{Structural Equation Models}
\label{subsec:LSEM}

Consider observational data in the form of a sample of independent copies of a random vector $X=(X_1, ... ,X_d)$ which, without loss of generality, is assumed to have zero mean.
\textit{Linear structural equation models} (LSEMs) assume that
$X$ solves the equation system 
\begin{equation} \label{LSEM}
X_j=\sum_{i\neq j}\beta_{ji}X_i + \epsilon_j, \quad j=1, ... ,d,
\end{equation}
where $B:=[\beta_{ji}]_{j,i=1}^d$ are unknown parameters that constitute  direct causal effects between the variables, and the $\epsilon_j$ are independent error terms with mean zero.  Following a line of work initiated by \cite{PetersEV}, we further assume the errors to be homoscedastic, that is, for an unknown variance parameter $\sigma^2\in(0,\infty)$ we have
\begin{equation}
\label{eq:equalvar}
\V[\epsilon_1]=\dots=\V[\epsilon_d]=\sigma^2.
\end{equation}
Each specific LSEM restricts a subset of the parameters $\beta_{ji}$ to be zero.  Put differently, each model is associated to a directed graph $\mathcal{G}$ and constrains $\beta_{ji}=0$ whenever $\mathcal{G}$ does not contain edge $i\to j$.
As in related work, we assume $\mathcal{G}$ to be a \textit{directed acyclic graph} (DAG).  Then the matrix $B$ is permutation similar to a strictly lower triangular matrix, and system \eqref{LSEM}  admits the unique solution $X=(I_d-B)^{-1} \epsilon$, where $I_d$ is the identity. Hence, $X$ has covariance matrix
\begin{equation}
\label{eq:cov-matrix}
    \E[XX^{T}]=\sigma^2(I_d-B)^{-1}(I_d-B)^{-T}.
\end{equation}

\subsection{Total Causal Effect}

In the causal interpretation of LSEMs the equations in \eqref{LSEM} are viewed as making assignments, with the variable on the left-hand side being assigned the value specified on the right-hand side. In this framework, the effect of an experimental intervention that externally sets the value of $X_i$ to $x_i$ is then captured by replacing the $i$th equation in \eqref{LSEM} by $X_i=x_i$.  In probabilistic notation this is expressed as $\text{do}(X_i=x_i)$, see \citet{Pearl09}.

Our interest is in the \emph{total causal effect} that an intervention on variable $X_i$ has on another variable $X_j$.  In linear models we may quantify this effect by considering a unit change in the intervention value $x_i$.  The total effect of $X_i$ on $X_j$ is then
\begin{equation*}\label{TCE}
    \mathcal{C}(i\rightarrow j):=\frac{\text{d}}{\text{d} x_i} \E[X_j|\text{ do}(X_i=x_i)] = 
    (I_d - B)^{-1}_{ji}.
\end{equation*}
Note that $\mathcal{C}(i\rightarrow j)=0$ if there does not exist a directed path from $i$ to $j$ in the underlying DAG $\mathcal{G}$.

Our goal in this article is to construct confidence intervals for $\mathcal{C}(i\rightarrow j)$ when the underlying causal structure is unknown and has to be learned.  We emphasize that this is a well-defined problem when the causal structure is identifiable, as in the homoscedastic setting laid out in Section~\ref{subsec:LSEM}. Indeed, in such a setting, every feasible distribution (or here simply, covariance matrix) uniquely determines a minimal causal graph which entails a unique value for $\mathcal{C}(i\rightarrow j)$.

Lacking alternative methods, a commonly used ``naive'' approach is to form confidence intervals by splitting the two involved tasks: First, the model, i.e., pattern of zero-entries of $B$, is estimated and, second, confidence intervals for the causal effects within the model are derived. This procedure, however, tacitly conditions away the uncertainty that arises from the data-driven model choice.  
For this reason, such naive confidence intervals often have poor coverage probabilities, especially under high uncertainty with respect to the model. For example, consider a bivariate example with a small true causal effect in one of the two possible causal directions. If the wrong causal ordering is learned, which happens with probability close to 0.5, one concludes with certainty yet wrongly that the causal effect is zero or formally one concludes, no matter the significance level, the confidence interval is the singleton set {0}. Hence, almost half the time the confidence interval will not cover the true non-zero parameter.

In the remainder of the paper we propose methods to address this issue.

\subsection{Resampling}
\label{sec:resampling}

Bootstrapping, subsampling and other resampling procedures are often applied to construct confidence intervals based on estimators whose sampling distribution is difficult to derive.  Resampling also offers a seemingly straightforward solution to the problem considered here.  This attempted solution proceeds by computing for each resampled data set a causal effect estimate that is obtained by concatenating a consistent model selection method that learns the graph and a consistent estimator of the causal effect in the learned model.  This concatenation is well defined by the identifiability of  homoscedastic LSEMs \citep{PetersEV,DrtonEV}, and it yields a function $T_{ij}$ such that $T_{ij}(\Sigma)=\mathcal{C}(i\rightarrow j)$ for all covariance matrices $\Sigma=\text{Var}[X_1,\dots,X_d]$ that come from a homoscedastic LSEM specified by \eqref{LSEM} and \eqref{eq:equalvar}.  However, bootstrap procedures that evaluate $T_{ij}$ on resampled data are not valid when there is non negligible uncertainty about the model.  Indeed, bootstrap/subsampling procedures may fail drastically when the mapping $T_{ij}$ lacks smoothness \citep{andrews:guggenberger:2010,drton:williams:2011}. Our simulations in Section~\ref{subsec:sims} validate those problems as can be seen in Table~\ref{tab::cover}, even in our simplified two-variable setting with generous sample size bootstrap methods do not achieve the required coverage probability. We emphasize that \citet{drton:williams:2011} demonstrate that for complex composite hypotheses settings (as in our case), the asymptotic behavior of bootstrap tests and confidence intervals is difficult to predict, even in low dimensions. 

The subtleties for homoscedastic LSEMs stem from the fact that the set of covariance matrices that are associated to at least one possible DAG is a union of smooth manifolds and singular at the intersections of the manifolds.  Figure~\ref{fig:d2} depicts this for  $d=2$ variables, where the two largest graphs $1\to 2$ and $1\leftarrow 2$ each define a 2-dimensional subset of the 3-dimensional cone of covariance matrices.  The singular points are where the two models meet, and their existence invalidates bootstrapping as a method for correctly capturing model uncertainty. This fact is again underlined by our simulations in Section~\ref{subsec:sims}, see Table~\ref{tab::cover}. Coverage failure occurs for small true causal effects and therefore further discredits bootstrapping as a method for calculating reliable confidence intervals for causal effects in practice.

\begin{figure}[t]
    \centering
    \includegraphics[scale=0.65]{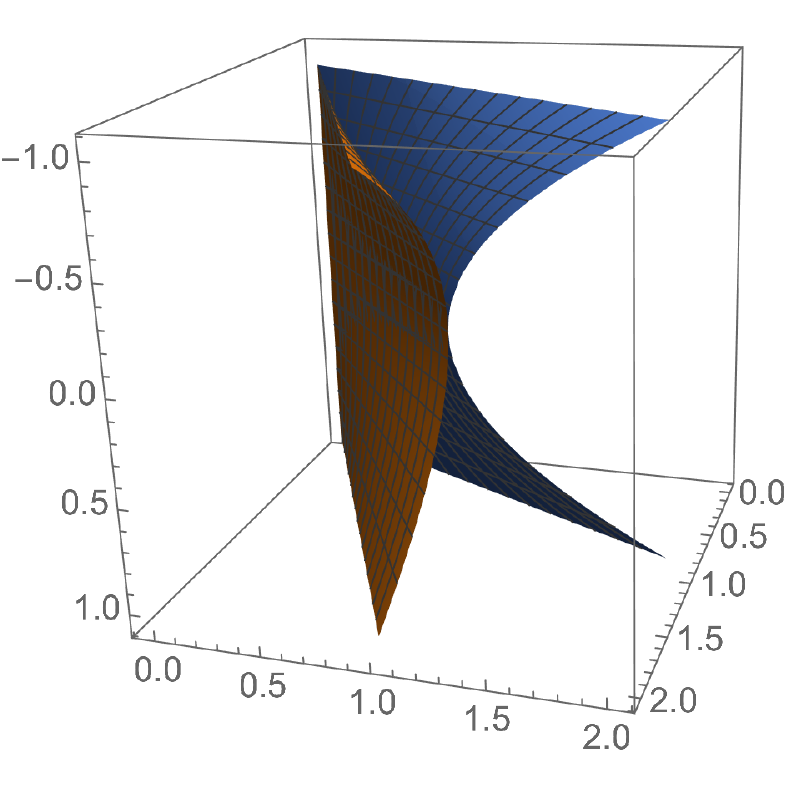}
    \caption{The set of $2\times 2$ covariance matrices coming from  LSEMs with homoscedastic errors.}
    \label{fig:d2}
\end{figure}

Let the cone of positive definite matrices be denoted by 
\begin{equation*}
    \M := \{\Sigma \in \R^{d\times d}\colon \Sigma^T = \Sigma,\, \Sigma \text{ positive definite}\}.
\end{equation*}
It is possible to define continuous extensions $\Tilde{T}:\M\to\mathbb{R}$ for $T_{ij}$.  However, in our exploration of such continuous extensions for $d=2$ variables, we were unable to give extensions for which there exists a unique scaling sequence $\tau_n$ such that $ \tau_n(\Tilde{T}(\hat{\Sigma})- \Tilde{T}(\Sigma))
$
always converges to a nondegenerate limit -- here, $\hat\Sigma$ is the empirical covariance matrix.  Convergence to nondegenerate limits is required for subsampling methods \citep{resampling}.

\subsection{Inversion of Tests}
\label{subsec:inversion}

In order to circumvent the difficulties posed by the non-smooth nature of the causal effect of interest, we develop in the following sections an approach that leverages the duality between statistical hypothesis tests and confidence regions.  Let $\alpha\in(0,1)$ be a fixed significance level.  Suppose that for each attainable value $\psi$ of the causal effect $\mathcal{C}(i\rightarrow j)$ we have a level $\alpha$ test of the hypothesis that the effect is indeed $\psi$.  Let $A(\psi)$ be the acceptance region, i.e., the set of all data sets for which the test does not reject $\psi$ as a hypothesized causal effect.  Then a $(1-\alpha)$-confidence region for $\mathcal{C}(i\rightarrow j)$ given the data set $x$ is
obtained as
\[C(x):=\{\psi :  x \in A(\psi)\}. \]

This approach shifts the burden to the construction of suitable tests of a hypothesized causal effect.  Without knowledge of the precise model, this remains challenging.  In the next section we present three concrete solutions based on likelihood ratio tests of order constraints \citep{Constrained}, and the recent theory of universal inference \citep{Universal}.

\subsection{Gaussian Likelihood}

In our construction of test statistics, we will assume the data to consist of random vectors $X^{(1)},\dots,X^{(n)}$ drawn independently from a (centered) multivariate normal distribution $\ND(0,\Sigma)$ with density $p(x| \Sigma)$.  Let $\hat\Sigma=\tfrac{1}{n}\sum_{i=1}^n X^{(i)}{X^{(i)}}^T$, and let $\M_0 \subset  \M$ be a subset of covariance matrices.  In order to test 
\begin{equation*}
	\text{H}_0: \Sigma \in\!\M_0
	\quad\text{against}\quad
	\text{H}_1: \Sigma \in\!\M \setminus \M_0 ,
\end{equation*}
we use the likelihood-ratio statistic
\[\lambda_n=2\Big(\sup_{\Sigma \in \M} \ell_n(\Sigma)- \sup_{\Sigma \in \M_0} \ell_n(\Sigma)\Big) \]
based on the log-likelihood $\ell_n(\Sigma)=\sum_{i=1}^n \log p(X_{(i)}|\Sigma)$ with
\begin{align*}
    \tfrac{2}{n}\ell_n(\Sigma)&=
-\log\det(2\pi\Sigma)-\text{tr}(\Sigma^{-1}\hat\Sigma).
    \end{align*}

\section{Bivariate Case}
\label{sec:bivariate}

This section develops the details of our approach in the two-dimensional setting, with two variables $X_1$ and $X_2$.  

\subsection{Representation of Causal Effect}\label{section:constraints}

For $d=2$, the model uncertainty boils down to uncertainty about the direction of the single edge, $1\to 2$ versus $1\leftarrow 2$, and we maintain the two possible LSEMs
\begin{align*}
\text{(M1)}\qquad &X_1=\epsilon_1,  &&X_2=\beta_{21} X_1  + \epsilon_2,\\
\text{(M2)}\qquad &X_1=\beta_{12} X_2 + \epsilon_1, &&X_2=\epsilon_2.
\end{align*}
The mere assumption of a homoscedastic LSEM imposes structure on the covariance matrix $\Sigma$ of $X=(X_1,X_2)^T$. Under model (M1), we obtain
\[ \Sigma= \begin{pmatrix}
\soo & \sot \\
\sot & \stt
\end{pmatrix}=\sigma^2 \begin{pmatrix}
1 & \beta_{21} \\
\beta_{21} & \beta_{21}^2 + 1
\end{pmatrix}, \]
which leads to the relations
\begin{equation}\label{M1-eq}
    \beta_{21}=\frac{\sot}{\soo},\quad\beta_{21}^2 + 1 = \frac{\stt}{\soo},\quad \soo^2 = \det(\Sigma),
\end{equation}
with $\det(\Sigma)= \soo \stt-\sot^2 $.
Analogously, for model (M2),
\begin{equation}\label{M2-eq}
    \beta_{12}=\frac{\sot}{\stt},\quad\beta_{12}^2 + 1 = \frac{\soo}{\stt},\quad  \stt^2 =  \det(\Sigma).
\end{equation}
Hence, the set of $2\times 2$ covariance matrices $\Sigma$ that are possible under homoscedastic LSEMs is $\M_r=\M_{r1}\cup\M_{r2}$, where
\begin{align}\label{Mr}
\M_{ra}:=\big\{&\Sigma \in \M \colon 
 \ \saa^2 = \det(\Sigma)\big\},\quad a=1,2. 
\end{align}

By symmetry, when considering total causal effects we may focus on the effect of $X_1$ on $X_2$, which is $\beta_{21}$ under model (M1) and zero under model (M2).  So,  by  \eqref{M1-eq},
\begin{equation}\label{ce-12}
    \mathcal{C}(1\rightarrow 2) =
    \frac{\sot}{\soo} \mathds{1}\{\Sigma\in\M_{r1}\}.
\end{equation}

\subsection{Constrained Likelihood-Ratio Tests}

Our construction of confidence sets inverts tests of hypotheses that specify ${\mathcal{C}(1\rightarrow 2) = \psi}$.  We now present two approaches to perform likelihood ratio tests.  Likelihood ratio statistics are easily defined but their probability distributions are generally difficult to determine at singularities, as encountered here where the alternative $\M_r$ is a union, and thus non-smooth.  To simplify distribution theory, we will relax the alternative to be the entire positive definite cone $\M$.

\subsubsection{Testing Inequality Constraints}

Our first approach exploits that  model selection for homoscedastic LSEMs can be achieved by ordering variances \citep{DrtonEV}.  Indeed, if $\Sigma\in\M_r$, then $\Sigma\in\M_{r1}$ precisely when $\soo\leq\stt$.  We will use this fact to set up hypotheses that encode ${\mathcal{C}(1\rightarrow 2) = \psi}$ for given $\psi$.  Three cases arise: 
$\psi=0$, $0 < \lvert \psi \rvert < 1$ and  $\lvert \psi \rvert \geq 1$.

\emph{Case $\psi=0$.}  Given $\Sigma\in\M_r$, we have  $\mathcal{C}(1\rightarrow 2) = 0$ if and only if \(\soo \geq \stt\). Hence, we conduct the test of
\begin{equation*}\label{H_0}
	\text{H}_0: \soo\geq\stt
	\quad\text{against}\quad
	\text{H}_1: \Sigma \in \M.
\end{equation*}
In this scenario, the asymptotic null distribution of $\lambda_n$ depends on the unknown value of $\Sigma$, but it is easy to see that the stochastically largest asymptotic distribution arises when $\soo=\stt$, in which case
\begin{equation*}
    \lambda_n \cd 0.5\chi_0^2 + 0.5\chi_1^2,\quad\text{as }n\to\infty,
\end{equation*}
where $\chi_d^2$ denotes a chi-square distribution with $d$ degrees of freedom and $\chi_0^2 \equiv 0$; see \cite{Constrained}.

\emph{Case $0 < \lvert \psi \rvert < 1$.}
When $\psi\not=0$, it must be that \(X_2\) is the causally dependent variable which corresponds to \(\soo \leq \stt\) and \(\sot/\soo = \psi \), according to (\ref{ce-12}). Hence, we test
\begin{equation*}
	\text{H}_0: \sot = \psi\soo, \,\soo\leq\stt
	\quad\text{against}\quad
	\text{H}_1: \Sigma \in \M.
\end{equation*}
For the least favorable covariance matrix in $\text{H}_0$, it holds that
\begin{equation*}
    \lambda_n \cd 0.5\chi_1^2 + 0.5\chi_2^2,\quad\text{as }n\to\infty.
\end{equation*}

\emph{Case $\lvert \psi \rvert \geq 1$.}  Again the two constraints \(\soo~\leq~\stt\) and \({\sot = \psi\soo}\) have to be satisfied.
However, the Cauchy-Schwarz inequality yields $\stt \geq \sot^2 /\soo = \psi^2\soo \geq \soo$, since \(|\psi| \geq 1\). Consequently, the inequality condition is automatically fulfilled and it suffices to test
\begin{equation*}
	\text{H}_0: \sot = \psi \soo
	\quad\text{against}\quad
	\text{H}_1: \Sigma \in \M.
\end{equation*}
The likelihood ratio statistic satisfies
\begin{equation*}
    \lambda_n\cd \chi_1^2,\quad \text{as } n\to\infty.
\end{equation*}

We remark that in the first and second case one may also follow a two-step procedure that uses a suitably calibrated pretest to decide which asymptotic distribution to employ \citep{two-step}.  We report no details on this approach here as we found the power gains to be only very slight.

Based on the above suite of tests, we may form a confidence interval from the accepted values of $\psi$, which we determine in practice by inspecting a fine grid of choices.  In our simulations in Section~\ref{subsec:sims} we refer to this method as \texttt{LRT1}.  We will also consider a heuristic variant in which we compute the likelihood ratio statistics by restricting the null and the alternative to the union of the two LSEMs, i.e., to $\M_r$, but still set critial values based on the asymptotic distributions given above. We refer to this method as \texttt{LRT1b}.

\subsubsection{Testing Polynomial Constraints}

The previous method encodes membership in model (M1) via the key inequality $\soo\le\stt$.  As an alternative we may directly work with the set of covariance matrices $\M_{r1}$ given in \eqref{Mr} when specifying null hypotheses.  However, to retain simple distributional approximations we continue to relax the alternative to be the entire p.d.-cone $\M$.

Recall that the causal effect $\mathcal{C}(1 \to 2)$ is non-zero only under model (M1), i.e., if $\Sigma\in\M_{r1}$.  In this case the effect is $\mathcal{C}(1 \to 2) = \sot/\soo$.  All matrices with $\mathcal{C}(1 \to 2)=0$ belong to $\M_{r2}$.  Thus, we test the null hypotheses
\begin{align*}
\text{H}_0: \begin{cases}
\sot = \psi \soo \text{ and }  \Sigma\in\M_{r1}, 
&\mathrm{if\,} \psi \neq 0, \\
\Sigma\in\M_{r2},
&\mathrm{if\,} \psi = 0. \\
\end{cases}
\end{align*}
We write $\Theta_0^{(\psi)}$ for the respective sets of covariance matrices.

\emph{Case $\psi\not=0$.}  If $\psi \neq 0$, the set $\Theta_0^{(\psi)}$ is a one-dimensional submanifold of the three-dimensional p.d.-cone $\M$.  The likelihood ratio statistic in this case is found to be
\begin{align*}
\lambda_n^{(\psi)}=& 2n \log\bigg\{\frac{1}{2\det( \hat\Sigma)^{1/2}}\mathrm{tr}\Big[
\begin{pmatrix}
1+\psi^2 & -\psi \\ -\psi & 1
\end{pmatrix} 
\hat\Sigma\Big]\bigg\}, 
\end{align*}
and for $\Sigma \in \Theta_0^{(\psi)}$ we have
\begin{align*}
\lambda_n^{(\psi)}  \cd \chi_2^2, \quad \text{as } n \to \infty.
\end{align*}

\emph{Case $\psi=0$.} The set $\Theta_0^{(0)}$ is a two-dimensional submanifold of $\M$ and yields the likelihood ratio statistic
\begin{align*}
\lambda_n^{(0)} = \,& 2n\log\bigg\{\frac{1}{2 \det (\hat\Sigma)^{1/2}}\Big(\hat\Sigma_{11} -\frac{\hat\Sigma_{12}^2}{\hat\Sigma_{22}} + \hat\Sigma_{22}\Big)\bigg\} .
\end{align*}
Here, for $\Sigma \in \Theta_0^{(0)}$,
\begin{align*}
\lambda_n^{(0)}  \cd \chi_1^2,\quad \text{as }n \to \infty.
\end{align*}
Given the explicit form of the likelihood ratio statistics we can explicitly determine an  asymptotic confidence set for the total causal effect.
\begin{theorem}
Let $\alpha \in (0,1)$. Then an asymptotic $(1-\alpha)$ confidence set for the causal effect $\mathcal{C}(1\rightarrow2)$ is given by
\begin{align*}
    C = \{\psi \neq 0  :  \lambda_n^{(\psi)}\leq \chi_{2,1-\alpha}^2 \} \cup \{0 : \lambda_n^{(0)}\leq \chi_{1,1-\alpha}^2\}.
\end{align*}
Furthermore, if we define
\[
K := 2\, \hat\Sigma_{11} \det (\hat\Sigma)^{1/2} \exp \Big(\frac{1}{2n} \chi^2_{2, 1-\alpha}\Big) - \hat\Sigma_{11}^2 - \det(\hat\Sigma) 
\]
and $K\geq0$, then 
\[\{\psi \neq 0  :  \lambda_n^{(\psi)}\leq \chi_{2,1-\alpha}^2 \}
=\left[\frac{\hat\Sigma_{12} - \sqrt{K}}{\hat\Sigma_{11}},\frac{\hat\Sigma_{12} + \sqrt{K}}{\hat\Sigma_{11}}\right].
\]
\end{theorem}

\begin{remark}
Since the dimension of $\Theta_0^{(0)}$ exceeds that of $\Theta_0^{(\psi)}$ for $\psi\not=0$, we are led to consider two different degrees of freedom for chi-square limits.  As a result there exist data for which we reject a zero effect but accept positive and negative effects that are arbitrarily small in magnitude.  However, this case arises very rarely.  If it does, it may be preferable to simply include zero in the confidence interval.
\end{remark}

\begin{proof}
If $\psi \neq 0$, it is easy to see that $\lambda_n^{(\psi)} \leq \chi^2_{2, 1-\alpha}$ if and only if
\begin{align*}\label{polynomial inequality lrt leq chi_1}
& \psi^2 \hat\Sigma_{11} - 2\psi \hat\Sigma_{12} + \hat\Sigma_{11}+ \hat\Sigma_{22}\\ &-2  \det( \hat\Sigma)^{1/2}\exp\Big(\frac{1}{2n}\chi^2_{2, 1-\alpha} \Big) \leq 0. 
\end{align*}
The inequality features a strictly convex quadratic polynomial in $\psi$.  The confidence interval is nonempty if the quadratic has real roots, which occurs for $K\ge 0$.  The roots  are $(\hat\Sigma_{12} \pm \sqrt{K})/\hat\Sigma_{11}$ and give the claimed explicit lower and upper end of the confidence interval.
The confidence set is then completed by checking whether we accept (and thus include) $\psi=0$.
\end{proof}
In our simulations, we refer to this method as \texttt{LRT2}.

\subsection{Split Likelihood Ratio Tests}

\citet{Universal} introduced the framework of universal inference, a general method for constructing hypothesis tests and confidence regions that are conservative but valid in finite samples.  
Universal inference employs a modification of the classical likelihood-ratio statistic termed the \textit{split likelihood ratio} (SLR), which is especially appealing for irregular composite hypotheses where asymptotic distributions are intractable.  As its name indicates, the SLR statistic is based on a data splitting approach.  Type-I error control is guaranteed by an application of Markov's inequality. 

The method proceeds by splitting the data into two subsets $D_0 = \{ X^{(1)},\ldots,X^{(k)} \}$ and $D_1 = \{ X^{(k+1)},\ldots,X^{(n)}\}$.  Let $\ell_0(\Sigma)$ denote the log-likelihood function based on $D_0$, that is, $\ell_0(\Sigma) = \sum_{i =1}^k \log p(X^{(i)} \mid \Sigma)$.
We then calculate the \textit{profile log-likelihood function} 
\begin{align*}
    \ell_\dagger(\psi) = \max\{ \ell_0(\Sigma): \Sigma\in\M_r,\; \mathcal{C}(1\to 2) = \psi\},
\end{align*}
and choose any estimator $\Tilde\Sigma^1$ based on $D_1$. Then
\begin{equation}\label{conf:split}
C= \{ \psi: \ell_0(\Tilde\Sigma^1)-\ell_\dagger(\psi) \leq \log(1/\alpha) \}
\end{equation}
is a (conservative) confidence set for the total causal effect $\mathcal{C}(1\to 2)$ with confidence level $1-\alpha$. 

Let $\hat\Sigma^0 =\frac{1}{k} \sum_{i=1}^k X^{(i)} {X^{(i)}}^T$ be the empirical covariance matrix for $D_0$.  In each of our two possible LSEMs, (M1) and (M2), maximizing $\ell_0$ with respect to the variance parameter $\sigma^2$ is straightforward; recall \eqref{eq:cov-matrix}.  We find that for fixed $B$, the maximum of $\ell_0$ over $\sigma^2>0$ is
\begin{align}\label{max loglikelihood sigma2}
- k\log\big\{\pi\,\mathrm{tr}\big[(I_d-B)^T (I_d-B) \hat\Sigma^0 \big]\big\} - k.
\end{align}
Now we assume that the causal effect $\mathcal{C}(1 \to 2)$ equals a fixed value $\psi\in\mathbb{R}$ and we maximize $\ell_0$ further over any remaining parameters.

\emph{Case $\psi \neq 0$.}  The covariance matrix $\Sigma$ is in $\M_{r1}$ and in its parametrization $\beta_{21}=\psi$. From \eqref{max loglikelihood sigma2}, we find that for $\psi \neq 0$, 
\begin{align}
\label{eq:max:s2}
\ell_\dagger(\psi)  = -k\log\bigg\{\pi\,\mathrm{tr}\left[
\begin{pmatrix}
1+\psi^2 & -\psi \\ -\psi & 1
\end{pmatrix} 
\hat\Sigma^0 \right]\bigg\} - k.
\end{align}

\emph{Case $\psi = 0$.}  Now, $\Sigma$ may be any matrix in $\M_{r2}$. 
Straightforward calculations show that 
\begin{align*} \max_{\beta_{12}}
\mathrm{tr}\left[\begin{pmatrix}
1 &-\beta_{12} \\ -\beta_{12}& 1+ \beta_{12}^2
\end{pmatrix}\hat\Sigma^0  \right]=
\hat\Sigma^0 _{11} - \frac{(\hat\Sigma^0_{12})^2}{\hat\Sigma^0_{22}} + \hat\Sigma^0_{22}.
\end{align*}
Inserting this expression in \eqref{max loglikelihood sigma2} yields the profile $\ell_\dagger(0)$.

With these preparations, we can now explicitly calculate the boundaries of the confidence set  for the total causal effect given in \eqref{conf:split}.
\begin{theorem} Let $\alpha \in (0,1)$ and define
\begin{align*}
G_a:=2& \hat\Sigma^0_{aa} \alpha^{-1/k}\det (\Tilde\Sigma^1)^{1/2} \exp\Big(\frac{1}{2} \mathrm{tr}[(\Tilde\Sigma^1)^{-1} \hat\Sigma^0] -1\Big)\\ &- (\hat\Sigma^0_{aa})^2 - \det(\hat\Sigma^0), \quad a=1,2.
\end{align*}
(i) If $G_1\geq0$, then the nonzero elements of the confidence set $C$ from \eqref{conf:split} are the nonzero elements of the interval $[L,U]$ with
 \[L:=\frac{\hat\Sigma^0_{12} - \sqrt{G_1}}{\hat\Sigma^0_{11}}, \quad  U:=\frac{\hat\Sigma^0_{12} + \sqrt{G_1}}{\hat\Sigma^0_{11}}. \]
(ii) The set $C$ from \eqref{conf:split} contains zero if and only if $G_2\geq0$.
\end{theorem}
\begin{proof}
Expanding the formula from \eqref{eq:max:s2}, we obtain that a value $\psi\neq 0$ satisfies the inequality for the confidence set in \eqref{conf:split} if and only if 
\begin{align*}\label{quadr ineq psi}
&\psi^2 \hat\Sigma^0_{11}- 2\psi \hat\Sigma^0_{12} +\hat\Sigma^0_{11} +\hat\Sigma^0_{22}  \\&- 2\alpha^{-1/k}(\det \Tilde\Sigma^1)^{1/2} \exp\Big(\frac{1}{2} \mathrm{tr}\big[(\Tilde\Sigma^1)^{-1} \hat\Sigma^0\big] -1 \Big)\leq 0. 
 \end{align*}
If $G_1\geq0$, the involved convex quadratic function has real roots at the claimed values of $L$ and $U$.  If $G_1<0$, the inequality has no solutions and the nonzero part of the confidence set remains empty.

The inequality that is equivalent to inclusion of $\psi = 0$ in the confidence set is similar.
\end{proof}

So far we have not made a specific choice for $\Tilde\Sigma^1$.  In \eqref{conf:split}, we observe that for small confidence sets it is desirable to form estimates $\Tilde\Sigma^1$ that achieve large values of $\ell_0(\Tilde\Sigma^1)$.  Since the universal inference approach poses no problems due to irregular geometry of hypothesis/alternative, it is natural to form an estimate that exploits the assumed validity of at least one of the two LSEMs.  In other words, we choose $\Tilde\Sigma^1$ to be the maximum likelihood estimator under the restriction that $\Sigma \in \M_r$.  We refer to this method as \texttt{SLRT}.

For problems with more than two variables determining the profile log-likelihood is computationally involved in the sense that higher degree polynomial equations need to be solved.  As an alternative we also report results for a heuristic in which we take $\Tilde\Sigma^1$ to be the unrestricted sample covariance matrix and replace the profile log-likelihood $\ell_\dagger(\psi)$ by $\ell_0$ evaluated at a consistent moment estimator.  To form this estimator under the constraint of an assumed causal effect $\mathcal{C}(i\to j)=\psi$, we consider every causal ordering that permits the effect, use the sample variance to estimate the variance of the first variable in the ordering, use sample covariances for all covariances of pairs of variables other than $(i,j)$ and fill the rest of the matrix so that the constraint $\mathcal{C}(i\to j)=\psi$ holds and all error variances are equal.  We then choose the ordering that leads to the maximum value of $\ell_0$. We refer to this method as \texttt{estSLRT}. 

\subsection{Simulations}
\label{subsec:sims}

In this section we present the results of a simulation study with the aim to compare the empirical coverage probabilities and widths of the different proposed confidence intervals for the causal effect. Our simulation experiment was designed as follows. We generated pseudo random numbers according to model (M1) ($1 \to 2$) and model (M2) ($1\leftarrow 2$), respectively, with standard normal errors.  For a selection of different values of $\beta\equiv\beta_{12}$ $(\beta\equiv\beta_{21})$ and different sample sizes $n$, we simulated  $10000$ independent data sets, for which we then determined the confidence sets for $\alpha=0.05$.  

The resulting empirical coverage probabilities are reported in Table \ref{tab::cover}, where all five proposed methods achieve the desired coverage frequency of $0.95$. Furthermore, all proposed methods seem to be particularly conservative if the causal effect is zero. In general the split likelihood ratio methods  are, as expected, the most conservative. For the purpose of comparison we included the empirical coverage probabilities of confidence intervals for the causal effect calculated with two different bootstrapping methods, as explained in Section~\ref{sec:resampling}. For each resampled data set \texttt{Bootstrap1} simply uses the sample covariance to select the direction (via (M1) if $\hat{\Sigma}_{11}\leq\hat{\Sigma}_{22}$, (M2) if $\hat{\Sigma}_{22}<\hat{\Sigma}_{11}$) and subsequently calculates the causal effect based on the selected model. \texttt{Bootstrap2} employs an established causal discovery algorithm for the equal variance case \texttt{GDS}, proposed by \citet{PetersEV}, to estimate the causal effect for each resampled data set. In principle the \texttt{GDS} method is a greedy search algorithm that maximizes the likelihood. As expected, and theoretically explained in Section~\ref{sec:resampling}, bootstrapping methods do not work in practice and do not achieve the the required coverage frequency.

\begin{table}
\centering
\caption{Empirical Coverage of 95\%-Confidence Intervals for the Total Causal Effect of $X_1$ on $X_2$ ($10000$ Replications).}
\resizebox{\linewidth}{!}{%
\begin{tabular}{ rr|rrrrr|rrrrr }
\toprule
 &  & \multicolumn{ 5 }{c}{ $X_1\to X_2$ }  & \multicolumn{ 5 }{|c}{ $X_2 \to X_1$ } \\ 
method & $n$\raisebox{0.1cm}{$\setminus\beta$}  & 0 & 0.05 & 0.1 & 0.2 & 0.5  &  0 & 0.05 & 0.1  & 0.2 & 0.5 \\ 
\midrule
\multirow{3}{*}{LRT1} & 100  & 1.00 & 0.95 & 0.95 & 0.96 & 0.98 & 1.00 & 1.00 &  0.99 & 0.98 & 1.00\\ 
& 500 & 1.00 & 0.95 & 0.95 & 0.96 & 0.97 & 1.00 & 0.99 &  0.98 & 0.98 & 1.00 \\
& 1000  & 1.00 & 0.96 & 0.96 & 0.97 & 0.98 & 1.00 & 0.98 & 0.97 & 0.99 & 1.00 \\
\midrule
\multirow{3}{*}{LRT1b} & 100 &  1.00 & 0.97 & 0.97 & 0.96 & 0.95  & 1.00 & 1.00 & 0.99 & 0.96 & 0.96 \\ 
& 500  & 1.00 & 0.97 & 0.97 & 0.96 & 0.98  & 1.00 & 1.00 & 0.99 & 0.94 & 1.00 \\ 
& 1000 & 1.00 & 0.97 & 0.97 & 0.95 & 0.97  & 1.00 & 1.00 & 0.99 & 0.94 & 1.00 \\ 
\midrule
\multirow{3}{*}{LRT2} & 100  & 0.97 & 0.96 & 0.96 & 0.96 & 0.96  & 0.96 & 0.97 & 0.97 & 0.97 & 0.97 \\ 
& 500  & 0.97 & 0.97 & 0.96 & 0.97 & 0.97  & 0.97 & 0.97 & 0.97 & 0.97 & 1.00 \\ 
&1000 & 0.97 & 0.96 & 0.96 & 0.96 & 0.96  & 0.96 & 0.96 & 0.96 & 0.97 & 1.00 \\
\midrule
\multirow{3}{*}{SLRT} & 100 & 1.00 & 1.00 & 1.00 & 1.00 & 1.00  & 1.00 & 1.00 & 1.00 & 1.00 & 1.00 \\ 
& 500 & 1.00 & 1.00 & 1.00 & 1.00 & 1.00  & 1.00 & 1.00 & 1.00 & 1.00 & 1.00 \\ 
& 1000  & 1.00 & 1.00 & 1.00 & 1.00 & 1.00 & 1.00 & 1.00 & 1.00 & 1.00 & 1.00 \\ 
\midrule
\multirow{3}{*}{estSLRT} & 100  & 0.99 & 0.98 & 0.98 & 0.98 & 0.98  & 0.99 & 0.99 & 0.99 & 0.99 & 0.99 \\ 
& 500 & 0.99 & 0.98 & 0.98 & 0.98 & 0.98  & 0.99 & 0.99 & 0.99 & 0.99 & 0.99 \\ 
& 1000  & 0.99 & 0.98 & 0.98 & 0.98 & 0.98  & 0.99 & 0.99 & 0.99 & 0.99 & 0.99 \\ 
\midrule
\multirow{3}{*}{Bootstrap1} & 100  & 0.99 & \textcolor{red}{0.89} & \textcolor{red}{0.89} & \textcolor{red}{0.91} & 0.94 &  1.00 & 1.00 & 0.99 & 0.99 & 1.00 \\ 
& 500  & 0.99 & \textcolor{red}{0.89} & \textcolor{red}{0.90} & 0.93 & 0.95 & 1.00 & 0.99 & 0.99 & 0.99 & 1.00 \\
& 1000  & 0.99 & \textcolor{red}{0.89} & \textcolor{red}{0.90} & 0.94 & 0.94 & 1.00 & 0.99 & 0.98 & 0.99 & 1.00 \\
\midrule
\multirow{3}{*}{Bootstrap2} & 100  & 1.00 & \textcolor{red}{0.56} &  \textcolor{red}{0.74} & 0.93 & 0.95 &  1.00 & 1.00 & 1.00 & 1.00 & 1.00 \\ 
& 500  & 1.00 &  \textcolor{red}{0.63} &  \textcolor{red}{0.84} & 0.94 & 0.95 &  1.00 & 1.00 & 1.00 & 0.99 & 1.00 \\ 
& 1000 & 1.00 &  \textcolor{red}{0.67} & 0.92 & 0.95 & 0.96 & 1.00 & 1.00 & 1.00 & 0.99 & 1.00 \\ 
\bottomrule
\end{tabular}%
}
\label{tab::cover}
\end{table}

Figure \ref{fig::width:n} displays the mean width of the smallest interval containing the constructed confidence set.  The widths are plotted against the sample size $n$ for a true causal effect of size $0.5$ (in different directions).  We note that while the confidence sets predominantly are intervals, it is possible that they are "torn" with $\{0\}$ as a disconnected component, reflecting the larger null hypothesis that is associated to a zero effect.
The more conservative split likelihood ratio methods yield wider confidence intervals. The \texttt{estSLRT} heuristic outperforms the standard \texttt{SLRT}.  As it does not fully optimize the profile log-likelihood function, \texttt{estSLRT} produces smaller sets, yet the desired empirical coverage is (easily) achieved.
In the case of no causal effect the confidence intervals converge to zero for all proposed methods. 

Figure \ref{fig::zero} shows the percentage of times zero is in the calculated confidence sets. The percentages are plotted for a total causal effect $\mathcal{C}(1\rightarrow 2)=0.5$ against the sample size and for sample size $n=500$ against the size of the causal effect. All proposed methods are consistent and exclude the possibility of no causal effect with increasing sample size. Therefore, all proposed methods not only yield correct confidence sets for the total causal effect but also successfully help decide whether the effect is nonzero or not.
Figure \ref{fig::zero} also shows that we exclude the possibility of no causal effect more frequently the higher the actual causal effect is. 

Even though it seems that in Figure \ref{fig::width:n} and \ref{fig::zero} the \texttt{LRT1b} method seems to perform best, we should stress that we only have theoretical guarantees for the methods \texttt{LRT1}, \texttt{LRT2} and \texttt{SLRT}. Out of those three methods the \texttt{LRT2} method seems to perform best, but this method also has to be handled with care as we will see in the following real data example.

\begin{figure}[t]
  \centering
  \includegraphics[width=\linewidth]{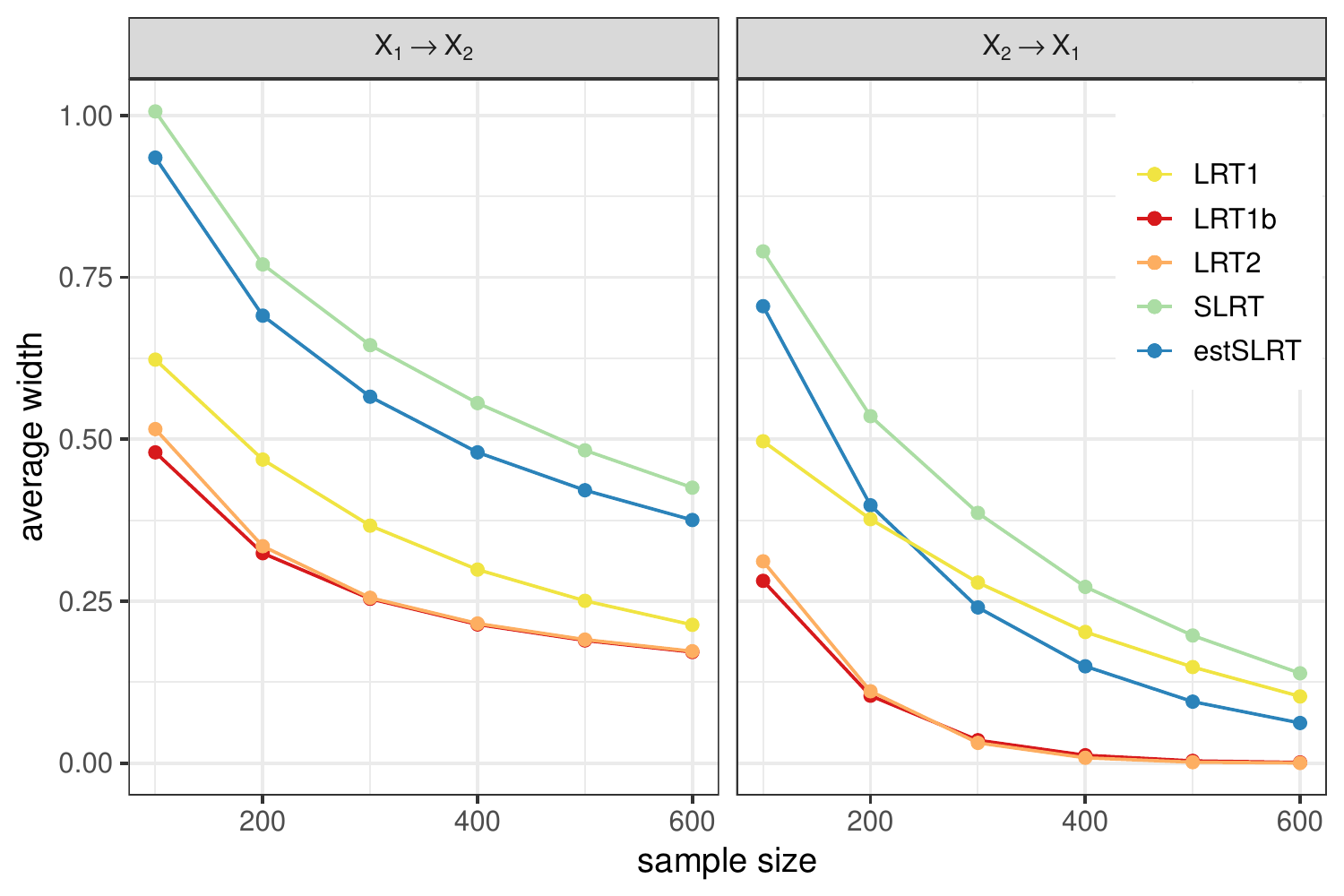}
  \caption{Average maximum width of 95\%-confidence intervals for the causal effect of $X_1$ on $X_2$ ($10000$ replications).}\label{fig::width:n}
\end{figure}

\begin{figure}
  \centering
  \includegraphics[width=\linewidth]{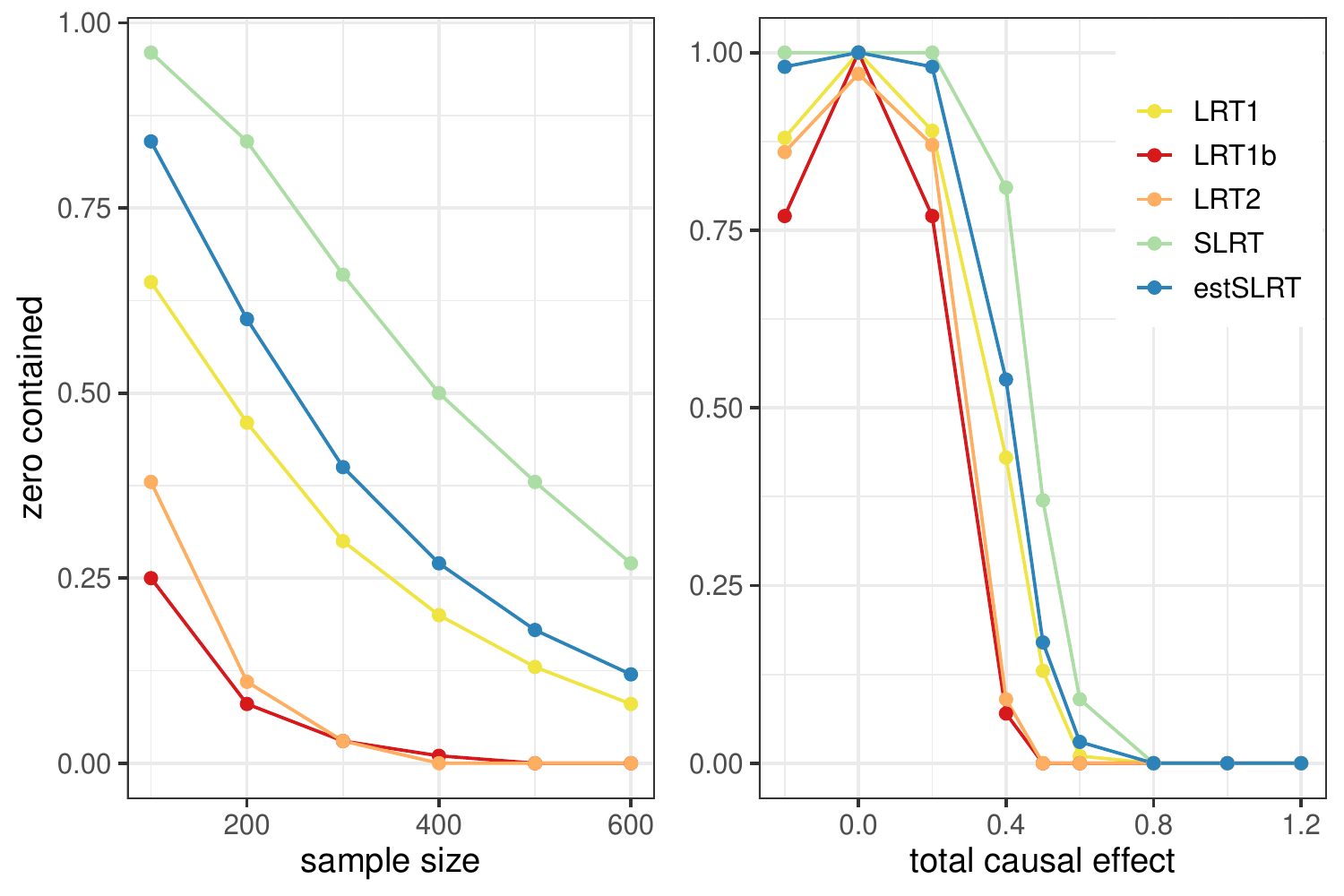}
  \caption{Percentage of times zero contained in the 95\%-confidence intervals for the total causal effect of $X_1$ on $X_2$ ($10000$ replications). (Left) $\mathcal{C}(1\rightarrow 2)=0.5$ against sample size. (Right) $n=500$ against size of the total causal effect.}
   \label{fig::zero}
\end{figure}

\subsection{Data Examples}
\label{subsec:data}

For a real world data benchmark we used the cause effect pairs data set presented in \citet{JMLR:Peters}.  It consists of different cause effect data pairs from various fields, for which the true causal direction is determined by domain knowledge.  For the application of homoscedastic LSEMs, we selected the following pairs: \texttt{pair66} and \texttt{pair67} containing daily stock returns, \texttt{pair76} containing the average annual rate of change of population and total dietary consumption, and \texttt{pair89} and \texttt{pair90} which describe the degree of root decomposition in forests and grasslands, respectively. The first three data sets exhibit a causal effect from $X_1$ to $X_2$ while the last two feature a causal effect from $X_2$ to $X_1$. 
Before calculating the confidence intervals for the causal effect of $X_1$ on $X_2$, we centered the data.

The results in Figure~\ref{fig::example} show that method \texttt{LRT2} which performed very well in the simulations produced an empty confidence set for all five data pairs.  In brief, the method is not able to cope with model misspecification (where we may still wish to obtain a confidence interval for a well-defined parameter).\footnote{E.g., one could consider the causal effect associated with the matrix $\Sigma\in\M_r$ such that the normal distribution $\ND(0,\Sigma)$ is closest in KL divergence to the data-generating distribution.}  The method \texttt{estSLRT} similarly suffers from this problem.  Indeed, both \texttt{LRT2} and \texttt{estSLRT} contrasts a linear and homoscedastic null hypothesis against a general Gaussian alternative.  Even for normal data, departures from homoscedasticity may favor the alternative in all testing problems that correspond to the confidence set, which then remains empty.

Although it was not the most statistically efficient method in our simulations, the method \texttt{LRT1} performs best for the real world data.  As the union of all its tested null hypotheses coincides with the alternative (the p.d.~cone $\M$), it always produces a nonempty confidence set.  Under misspecification the interval targets the parameter
\[
\tilde{\mathcal{C}}(1\rightarrow 2) =
    \frac{\sot}{\soo} \mathds{1}\{\Sigma_{11}\le \Sigma_{22}\};
\]
an extension of the parameter defined in \eqref{ce-12}.  In this sense the method is less sensitive to departures from homoscedasticity, or even linearity if the true covariance matrix $\Sigma$ is defined as furnishing the KL-best normal approximation to the data-generating distribution.  The interval width for \texttt{LRT1} is here similar or even smaller than for the heuristic \texttt{LRT1b}.

The standard split likelihood ratio method \texttt{SLRT} also performs well, but produces considerably wider intervals than \texttt{LRT}.  The width of the estimated confidence intervals with the split likelihood ratio methods slightly vary depending on how the real data set is (randomly) split.

Finally, we note that all proposed methods recognize that there is no causal effect from $X_1$ to $X_2$ in the last two data pairs.  For the first data pair, some uncertainty remains at the 95\%-level  about whether an effect is indeed present (\texttt{LRT1} and \texttt{SLRT}). For comparison, we also calculated the confidence intervals for the causal effect of $X_2$ on $X_1$, see Figure~\ref{fig::example::swap}, and the results lead to the same conclusions.

\newpage

\begin{figure}
  \centering
  \includegraphics[width=\linewidth]{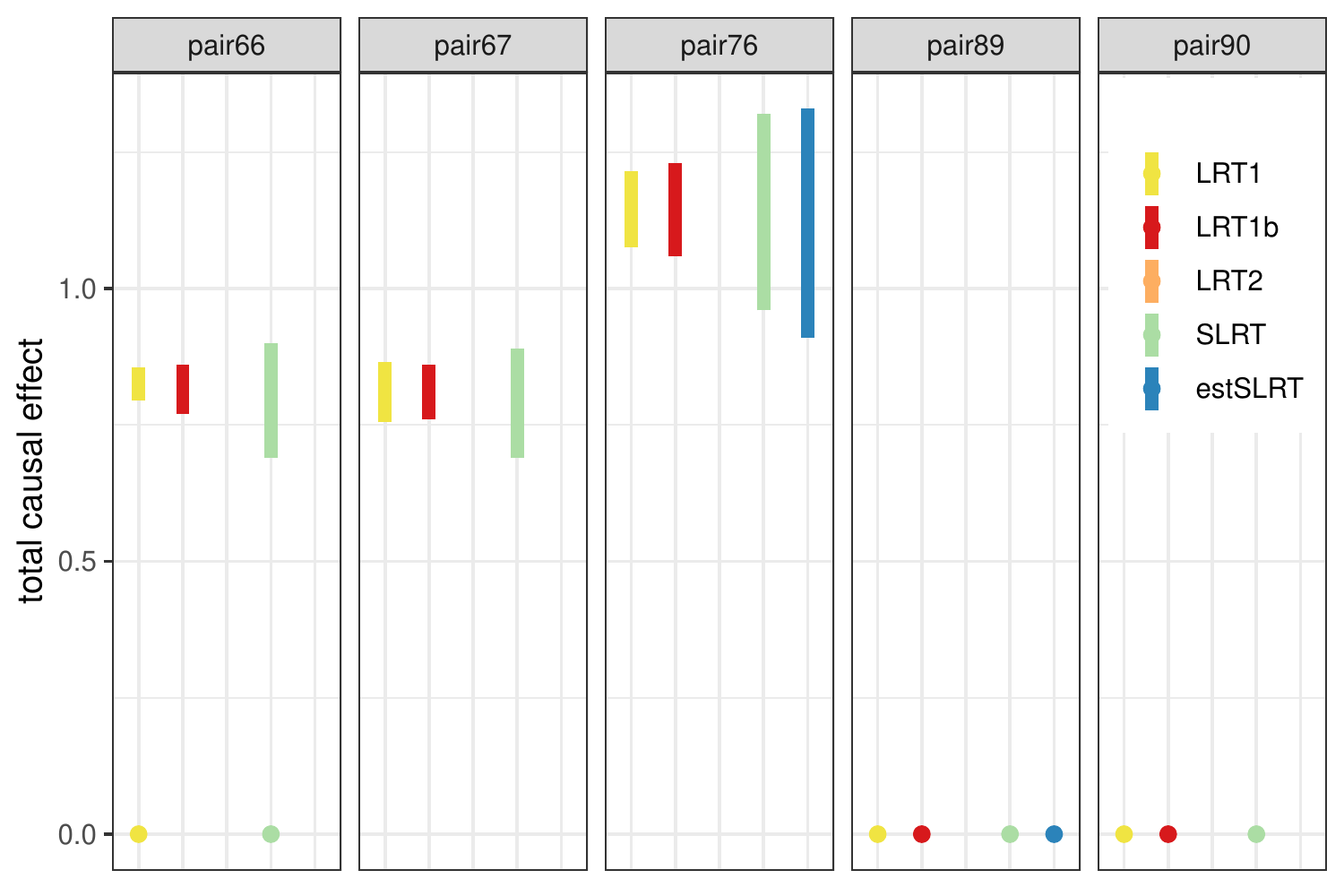}
  \caption{95\%-Confidence intervals for the total causal effect of $X_1$ on $X_2$  for different real world data pairs.}\label{fig::example}
\end{figure}

\begin{figure}
  \centering
  \includegraphics[width=\linewidth]{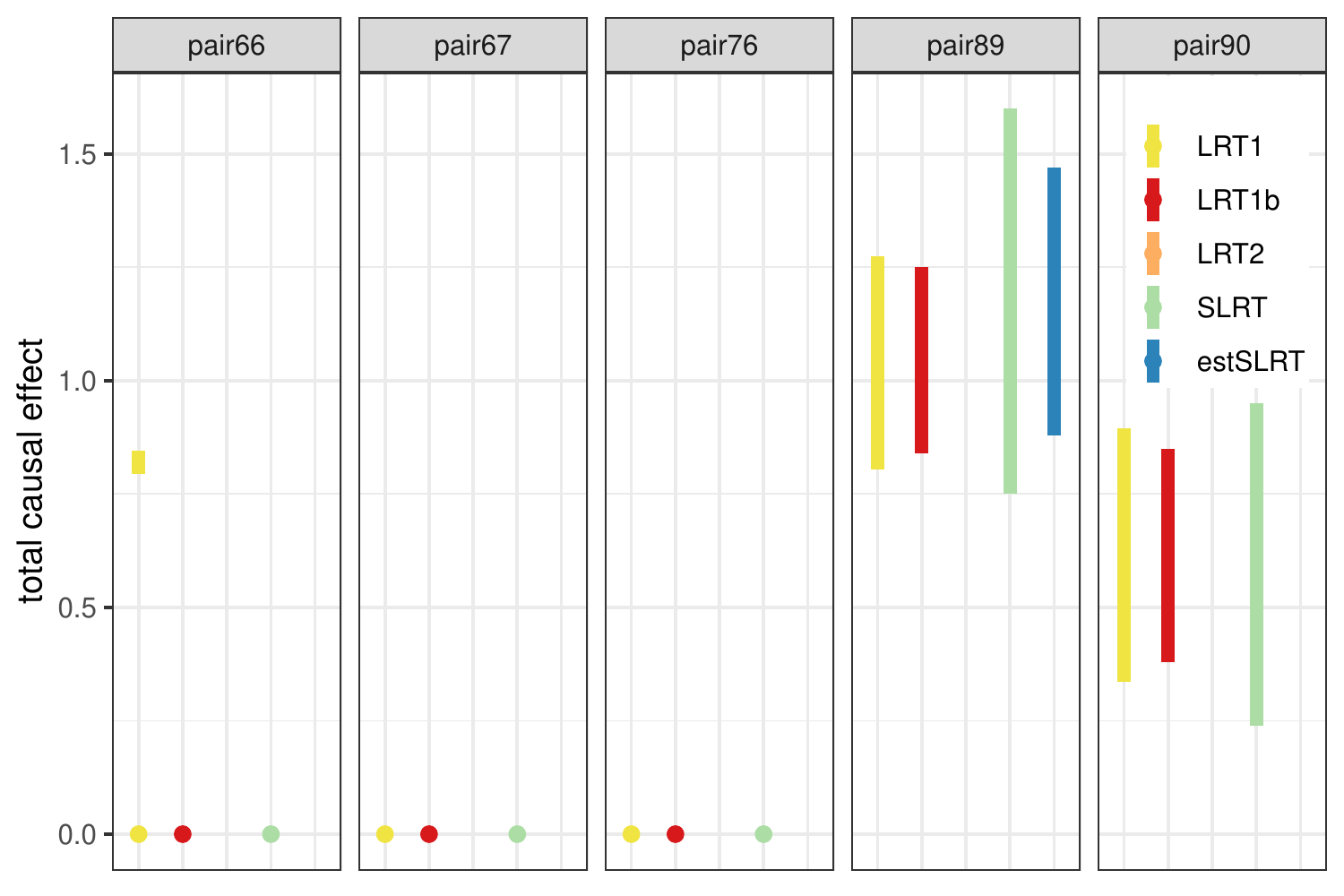}
  \caption{95\%-Confidence intervals for the total causal effect of $X_2$ on $X_1$  for different real world data pairs.}\label{fig::example::swap}
\end{figure}

\section{Higher Dimensions}
\label{sec:higher-dim}

In this section we give an outlook on how one may extend the proposed methods to higher dimensional cases.  The asymptotic distribution of the first  method \texttt{LRT1} via testing inequalities for the conditional variances is a mixture of chi-square distributions. The calculation of those mixture weights is difficult in higher dimensions.  Using data-dependent critical values as used in \cite{almohamad:2020} may be an option to push the methodology to moderately small dimension, but at this point we have not yet explored this option.

Calculating asymptotic distributions for the second method \texttt{LRT2} remains feasible in any individual LSEM that allows for a given effect $\mathcal{C}(i\to j)$ to be nonzero.  However, one has to then address the issue that several LSEM allow for $\mathcal{C}(i\to j)\not=0$ and the relevant null hypothesis becomes a union of smooth manifolds.  One way to address this problem would be to form an intersection union test, but this is again a topic for future work.

The approach that is the simplest for extension to higher dimensional cases are the split likelihood ratio tests.  We illustrate this in the following simulations for the case of $d=3$ variables. The three dimensional case allows for six possible models based on the ordering of the three involved variables $X_1, X_2$ and $X_3$. Figure~\ref{fig::width:high} displays the average maximum width and empirical coverage of the estimated confidence intervals for the total causal effect of $X_1$ on $X_2$ in the six different models. In the upper three cases there is a true causal effect of size $0.5$, while in the other cases there is no causal effect of $X_1$ on $X_2$. The constructed confidence intervals have a high coverage, exceeding the desired coverage with the method \texttt{SLRT} in all cases and with the heuristic method \texttt{estSLRT} in five of the six possible cases.  

\begin{figure}
  \centering
  \includegraphics[width=\linewidth]{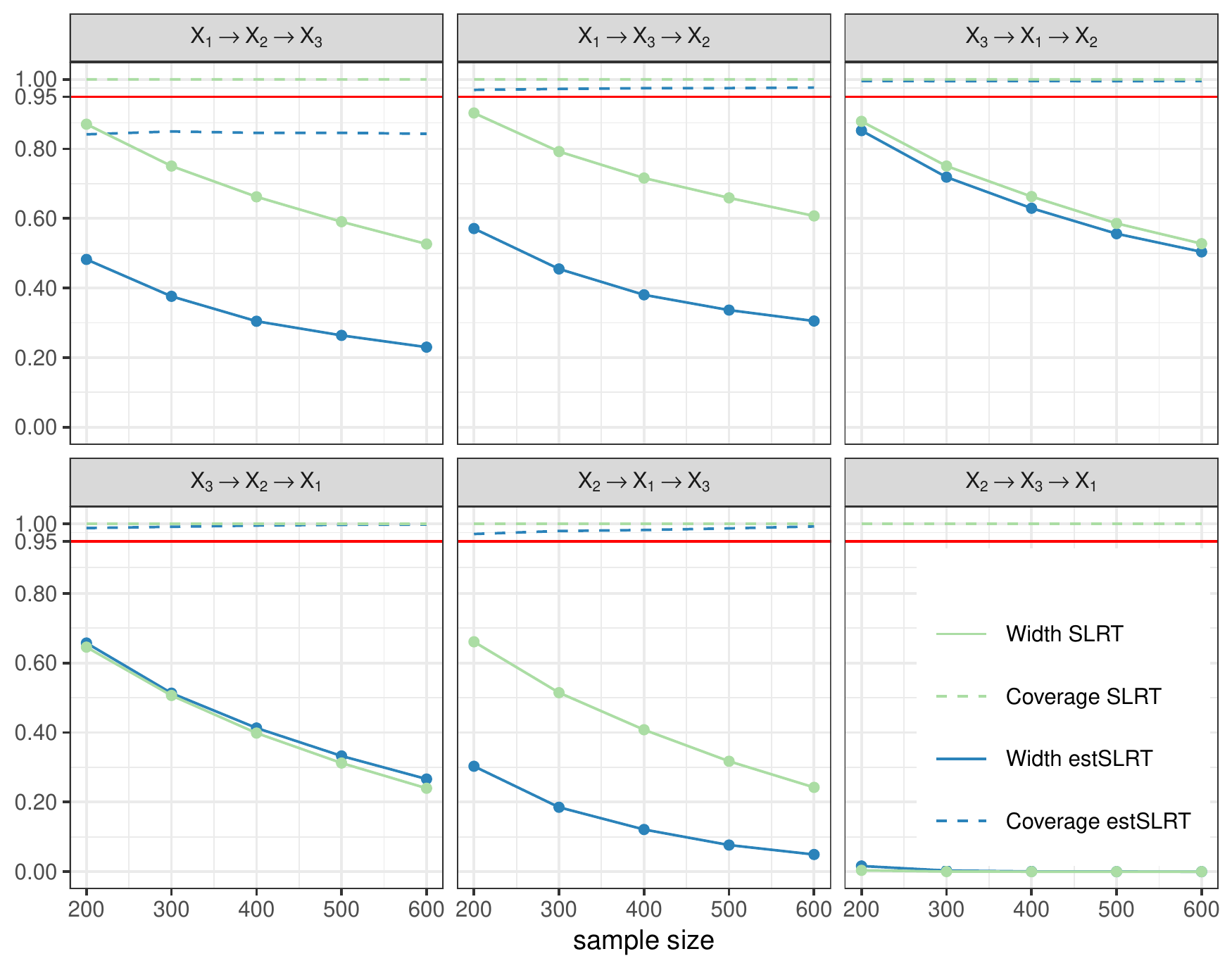}
  \caption{Empirical coverage and average maximum width of the 95\%-confidence interval for the total causal effect of $X_1$ on $X_2$ ($10000$ replications).}\label{fig::width:high}
\end{figure}

\section{Conclusion} 
\label{sec:discussion}

We proposed new methods to construct confidence intervals for the total causal effect in problems in which causal structure is unknown but identifiable.   We cope with this uncertainty in a test inversion approach that accounts for both types of uncertainty: causal structure and numerical size of nonzero effects.  For two-variable problems the empirical results for the \texttt{LRT1}  method that tests inequalities among variances are very promising, but it may prove difficult to extend this method to higher dimensional cases and settings other than homoscedastic LSEMs.  The second proposed \texttt{LRT2} method tests the polynomial constraints imposed by the LSEM assumption and can be extended to higher dimensional cases.  However, it can be  sensitive to departures from the modeling assumptions, as seen in our real data example.  This is due to the fact that the precise model assumptions were incorporated in the null hypotheses, but for simplicity in distribution theory not in the alternative.  An interesting problem for further research would be to improve our understanding of possible asymptotic approximations for the \texttt{LRT2} statistics when the alternative is not relaxed but kept as the union of all homoscedastic LSEMs.

\begin{figure}
  \centering
  \includegraphics[width=\linewidth]{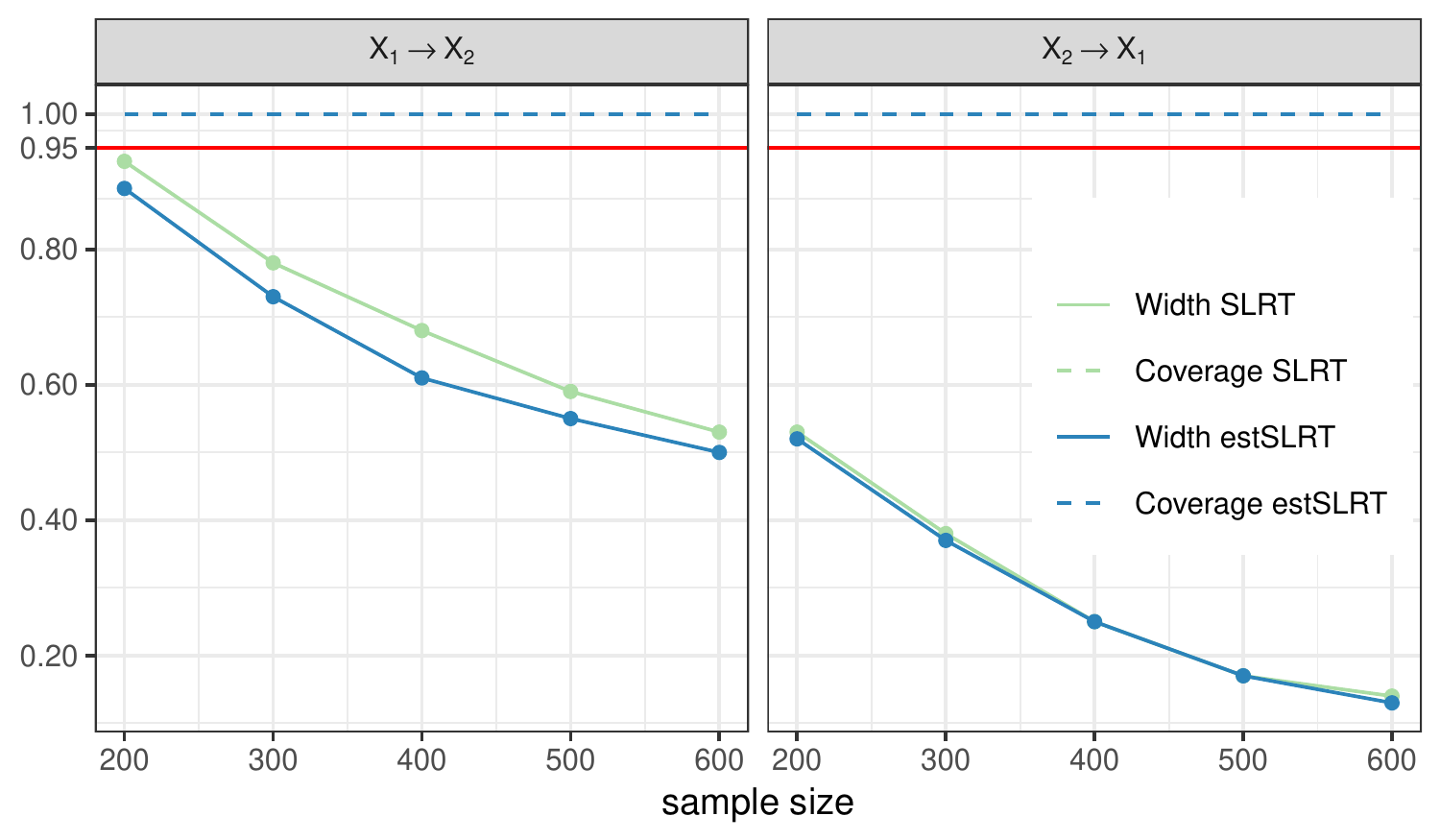}
  \caption{Empirical coverage and average maximum width of the 95\%-confidence interval for the total causal effect of $X_1$ on $X_2$ in a LiNGAM setting ($100$ replications).}\label{fig::width:lingam}
\end{figure}

The last proposed \texttt{SLRT} methods based on the theory of universal inference are the most conservative but also easiest to apply methods.  They can be extended rather directly not only to higher dimensional cases but also to other modeling frameworks.  To illustrate the latter point, we briefly consider the usage of the split likelihood ratio methods for LSEMs with non-Gaussian errors (LiNGAM). \citet{Shimizu} showed that under these assumptions unique identification is possible and the causal structure imposes constraints on the (conditional) moments, see \citet{DrtonLingam}. We can thus use empirical likelihood methods \citep{wang:2017} to form a split likelihood ratio and construct confidence intervals for the causal effect. Figure \ref{fig::width:lingam} shows the empirical coverage and average maximum width of these confidence intervals for the causal effect of $X_1$ on $X_2$ in the bivariate LiNGAM setting. We simulated the data with a causal effect of size $0.5$ (in different directions) and uniform$(-1,1)$ distributed error terms.  We observe that the method is conservative yet informative.

\begin{acknowledgements} 
This project has received funding from the European Research Council (ERC) under the European Union’s Horizon 2020 research and innovation programme (grant agreement No 883818).
\end{acknowledgements}

\bibliography{strieder_473}
\end{document}